\documentclass[a4paper,UKenglish,cleveref, autoref, thm-restate, pdfa]{lipics-v2021}

\nolinenumbers
\hideLIPIcs

\graphicspath{{./img/}}

\title{Packing Odd Walks and Trails in Multiterminal Networks}

\author{Maxim Akhmedov\footnote{Corresponding author}}{Moscow State University, Department of Mathematical Logic and Algorithms, Moscow, Russia}{}{https://orcid.org/0000-0002-7947-1416}{}

\author{Maxim Babenko\footnote{Corresponding author}}{Higher School of Economics, Moscow, Russia}{}{https://orcid.org/0000-0003-2877-9725}{}

\authorrunning{M. Akhmedov and M. Babenko}

\Copyright{Maxim Akhmedov and Maxim Babenko}



\EventEditors{Petra Berenbrink, Mamadou Moustapha Kant\'{e}, Patricia Bouyer, and Anuj Dawar}
\EventNoEds{4}
\EventLongTitle{40th International Symposium on Theoretical Aspects of Computer Science (STACS 2023)}
\EventShortTitle{STACS 2023}
\EventAcronym{STACS}
\EventYear{2023}
\EventDate{March 7--9, 2023}
\EventLocation{Hamburg, Germany}
\EventLogo{}
\SeriesVolume{254}
\ArticleNo{26}

\usepackage[usenames,dvipsnames]{xcolor}
\usepackage{mathtools}
\usepackage{calc}
\usepackage{bbm}
\usepackage[color=Dandelion,textsize=footnotesize,colorinlistoftodos,textwidth=\dimexpr\marginparwidth-0.4cm\relax]{todonotes}

\newcommand{\ZZ}{\mathbb{Z}}
\newcommand{\RR}{\mathbb{R}}

\newcommand{\wt}{\widetilde}
\newcommand{\val}[1]{\Vert #1 \Vert}
\newcommand{\1}{\mathbf{1}}
\newcommand{\2}{\mathbf{2}}

\makeatletter
\newcommand*\bigcdot{\mathpalette\bigcdot@{.5}}
\newcommand*\bigcdot@[2]{\mathbin{\vcenter{\hbox{\scalebox{#2}{$\m@th#1\bullet$}}}}}
\makeatother

\begin{document}

\maketitle


\begin{abstract}

Let $G$ be an undirected network with a distinguished set of \emph{terminals} $T \subseteq V(G)$ and edge \emph{capacities} $cap: E(G) \rightarrow \RR_+$. By an \emph{odd} $T$-walk we mean a walk in $G$ (with possible vertex and edge self-intersections) connecting two distinct terminals and consisting of an odd number of edges. Inspired by the work of Schrijver and Seymour on odd path packing for two terminals, we consider packings of odd $T$-walks subject to capacities $cap$.

First, we present a strongly polynomial time algorithm for constructing a maximum fractional packing of odd $T$-walks. For even integer capacities, our algorithm constructs a packing that is half-integer. Additionally, if $cap(\delta(v))$ is divisible by 4 for any $v \in V(G) - T$, our algorithm constructs an integer packing.

Second, we establish and prove the corresponding min-max relation.

Third, if $G$ is inner Eulerian (i.e. degrees of all nodes in $V(G) - T$ are even) and $cap(e) = 2$ for all $e \in E$, we show that there exists an integer packing of odd $T$-trails (i.e. odd $T$-walks with no repeated edges) of the same value as in case of odd $T$-walks, and this packing can be found in polynomial time.

To achieve the above goals, we establish a connection between packings of odd $T$-walks and $T$-trails and certain multiflow problems in undirected and bidirected graphs.
\end{abstract}


\section{Introduction}
\label{sec:intro}

Hereinafter, for graph $G$ we use notation $V(G)$ (resp. $E(G)$) to denote the set of vertices (resp. edges) of~$G$.

Consider an undirected network $G$ with a distinguished set of \emph{terminals} $T \subseteq V(G)$ and edge capacities $cap: E(G) \rightarrow \RR_+$. We use the notions of \textbf{walks} and \textbf{paths}; the former allow arbitrary edge and vertex self-intersections, while the latter forbid any self-intersections. Additionally, we consider \textbf{trails} that allow vertex self-intersections but not edge self-intersections. Any path is a trail and any trail is a walk, but not vice versa. A \textbf{$T$-walk} (resp. \textbf{$T$-trail} or \textbf{$T$-path}) is a walk (resp. trail or path) connecting two distinct vertices in $T$ (note that its intermediate vertices may also be in $T$).

By a (fractional) \textbf{packing} of $T$-walks (resp. $T$-trails, $T$-paths) subject to capacities $cap$ we mean a weighted collection $\mathcal{P} = \{\alpha_1 \cdot W_1, \ldots, \alpha_m \cdot W_m\}$, where $W_i$ are $T$-walks (resp. $T$-trails, $T$-paths) and $\alpha_i \in \RR_+$ are \textbf{weights} such that $\sum_i \alpha_i n_i(e) \le cap(e)$ for any $e \in E(G)$, where $n_i(e) = 0, 1, \ldots$ denotes the number of occurrences of $e$ in $W_i$. If all $\alpha_i$ are integer (resp. $\frac1k$-integer, i.e. become integer after multiplying by $k$) then the whole packing is said to be \textbf{integer} (resp. \textbf{$\frac1k$-integer}). The \textbf{value} of $\mathcal{P}$ (denoted by $\val{\mathcal{P}}$) is $\sum_i \alpha_i$; a packing of maximum value will be referred to as \textbf{maximum}.

If one imposes no additional restrictions, the values of maximum packings of $T$-walks, $T$-trails and $T$-paths coincide; this follows from the fact that any walk can be reduced into a path by removing its cyclic parts. Also for $|T| = 2$ and integer edge capacities the value of a maximum fractional packing equals the value of a maximum integer packing (by the max-flow integrality theorem \cite[Cor.~10.3a]{Sch-03}), while for $|T| \ge 3$ a maximum packing may be half-integer~\cite[Sec.~73.2]{Sch-03}.

Now consider a much harder case where $T$-walks (resp. $T$-trails, $T$-paths) comprising a packing are required to be \textbf{odd}, i.e. to consist of an odd number of edges. Now an attempt to transform a walk into a path (or even a walk into a trail) by a similar decycling approach fails since it may alter the parity.

Our original source of inspiration lies in the work of Schrijver and Seymour \cite{SS-94}, who established a min-max formula for the value of a maximum fractional packing of odd $T$-paths for $T = \{s, t\}$. At its dual side, the formula involves enumerating (not necessarily induced) subgraphs $H$ of $G$ that contain both $s$ and $t$ but no odd $s-t$ path and upper-bounding the value of packings by a certain ``capacity'' of $H$. In a sense, such $H$ is analogous to an $s-t$ cut in the standard max-flow-min-cut-theorem \cite[Th.~10.3]{Sch-03} and is called an \textbf{odd path barrier}.

The above result is established for just $|T| = 2$, only concerns fractional packings and, moreover, is non-constructive. In case of integer capacities one should ultimately aim for a min-max formula and a polynomial algorithm for constructing a maximum integer packing of odd $T$-paths. These questions, unfortunately, seem to be notoriously hard. In particular, \cite[Sec.~3.3]{Yam-16} shows that checking if a given graph contains a pair of edge-disjoint odd $T$-trails is NP-hard already for $|T| = 2$.

\subsection{Our results}
The present paper deals with the \textbf{multiterminal} version of the problem (allowing arbitrary number of terminals $T$) but considers packings of odd $T$-walks and $T$-trails rather than odd $T$-paths.

\textbf{First}, for packings of odd $T$-walks and real-valued capacities we present a polynomial time reduction (\cref{thm:odd_walk_packing}) from a maximum odd $T$-walk packing problem to a maximum multiflow problem for a special commodity graph family due to Karzanov \cite{Kar-94}. For even integer capacities, our algorithm produces a half-integer packing. Also, if capacities are even integers and $cap(\delta(v))$ (which is, as usual, the sum of $cap(e)$ for all edges $e$ incident to $v$) is divisible by 4 for any $v \in V(G) - T$, a maximum packing can be made integer.

\textbf{Second}, we present a min-max formula (\cref{thm:min_max_odd_walk}) for maximum odd $T$-walk packings. It is strikingly similar to the one due to Schrijver and Seymour \cite{SS-94} (for odd $s-t$ paths) and involves, at its dual side, subgraphs of $G$ containing no odd $T$-walks.

\textbf{Third}, we extend the above results to odd $T$-trail packings. Consider the unit-capacity case. Then a Schrijver--Seymour-type min-max relation does not hold for integer packings even if one assumes that $|T| = 2$ and the underlying graph is \textbf{inner Eulerian} (i.e. degrees of all non-terminal verticies are even). An example of such a ``bad'' instance can be found in \cite[Sec.3]{SS-94}. (There it was given for the case of odd $T$-paths rather than odd $T$-trails but it turns the example works in both cases.)

The fractionality status of such a packing problem seems to be open. We partially resolve it by proving that for an inner Eulerian graph with unit capacities and an arbitrary number of terminals an optimum packing of $T$-trails can always be chosen half-integer (and also can be found in polynomial time). If all capacities are multiplied by 2, the inner Eulerianness condition becomes $cap(\delta(v))$ being divisible by 4 for all $v \in V(G) - T$, which is equivalent to the condition from the first result, and our optimum packing becomes integer.

We prove (\cref{thm:odd_trail_packing}) that there exists a packing of odd $T$-trails of the same value as in the case of odd $T$-walks, and this packing can be found in polynomial time. In other words, odd $T$-walks forming a maximum integer packing can always be rearranged (``untangled'') to ensure that none of them has edge self-intersections.

\subsection{Our techniques}
The algorithm that deals with odd $T$-walks is based on a reduction to a certain multiflow problem \cite{Kar-94} and some graph symmetrization. For the min-max formula regarding packing of odd $T$-walks, we indicate how optimum collections of cuts (in the sense of the above multiflow problem) correspond to minimum odd $T$-walk barriers.

The algorithm dealing with odd $T$-trails attracts additional combinatorial ideas. Loosely speaking, it constructs a maximum integer packing consisting of odd $T$-walks~$W_i$. If all of these $T$-walks $W_i$ are already $T$-trails (i.e. have no edge self-intersections), then we are done. Otherwise, for walk $W_i$ with edge self-intersections, we either simplify $W_i$ (while maintaining its parity) or find a \textbf{redundant} edge in $G$ whose removal does not decrease the number of $T$-walks in the current packing, drop it, and repeat. The existence of a redundant edge is proved by a novel characterization of integer odd $T$-walk packings in terms of $T$-trail packings in inner Eulerian \textbf{bidirected} graphs~\cite{BK-07} and relies on the corresponding min-max theorem.

\subsection{Related work}
There is also a solid body of recent research devoted to path packings in unit-capacitated graphs. (Note that here the notions of integer walk and trail packings coincide.)

While even for $T = \{s, t\}$ the problem of finding a maximum integer packing of odd $T$-trails in general networks does not seem to be tractable, a certain lower bound for the maximum value of such packings (relating it to odd $T$-trail covers) is known \cite{IS-17} (also see \cite{CMW-16} for a weaker bound).

Note that if one is interested in packing odd $T$-trails (rather than odd $T$-walks) in graphs with integer capacities larger than~1, then the problem does not seem to be directly reducible to unit capacities. Indeed, splitting each edge and solving the problem in the unit-capacitated case, one will face challenges with edge self-intersections when attempting to return back to the original graph.

These challenges seem to be quite fundamental, and, in particular, we are not aware of any prior art concerning capacitated versions of the maximum odd $T$-trail integral packing problem.  Our algorithm for constructing a maximum packing of odd $T$-trails is able to deal with edge self-intersections by certain $T$-walk ``untangling'' but this battle is not won easily.

Another related (but still substantially different) area of research concerns integer packing of vertex-disjoint $A$-paths in \textbf{group-labeled} graphs. Here each edge $xy \in E$ is endowed with an element $g(x,y)$ of group~$\Gamma$ (obeying $g(x,y) = -g(y,x)$). Path $P$ with both (distinct) ends in $A \subseteq V(G)$ is called a \textbf{non-zero $A$-path} if the sum of all group elements corresponding to (directed) edges of $P$ is non-zero. (This also extends to non-Abelian groups.) In \cite{CGC-08} a polynomial algorithm for constructing a maximum integer packing of vertex-disjoint $A$-paths is given. See also~\cite{Pap-08} for a similar treatment involving permutation groups.

With an appropriate choice of group~$\Gamma$ and edge labels, non-zero $A$-paths may express various well-studied notions, e.g. the much-celebrated Mader's integer packings of vertex-disjoint $\mathcal{S}$-paths \cite{Mad-78}, \cite[Sec.~73.1]{Sch-03}.

Note that if $\Gamma = \ZZ_2$ and $g(x,y) = 1$ for all edges $xy$ one gets the odd parity constraint for paths comprising a packing. The latter motivates adding such $\Gamma$ as a direct group summand in the Mader's case above hoping to capture the parity restriction. This approach, however, will not work as expected: now a path could either be odd \emph{or} connect terminals in distinct $\mathcal{S}$-classes (while we were certainly hoping for paths that simultaneously have ends in distinct $\mathcal{S}$-classes \emph{and} have odd length).


\section{Walks, trails, packings and other notation}

Consider an undirected loopless graph $G$ with possible parallel edges. In this paper we deal with certain families of path-like objects in $G$ differing in kinds of allowed self-intersections. Formally:

\begin{definition}
\label{def:walk_def}
Given $x, y \in V(G)$, an $x-y$ \textbf{walk} is a sequence $W = (e_1, e_2, \dots, e_l)$, where $e_i \in E$ are such that $e_i = v_{i-1} v_{i}$ for $v_0 = x$, $v_l = y$ and $v_1, \dots, v_{l-1} \in V(G)$.

Here $l$ is called the \textbf{length} of $W$. A walk is called \textbf{even} or \textbf{odd} depending on the parity of its length. Vertices $x$ and $y$ are called the \textbf{endpoints} of $W$ and $v_1, \dots, v_{l-1}$ are called \textbf{intermediate} (for $W$).
\end{definition}

Note that some of vertices $v_i$ of $W$ may coincide, allowing a walk to visit the same vertex multiple times and traverse same edge multiple times.

\begin{definition}
An $x-y$ \textbf{trail} (resp. \textbf{path}) is an $x-y$ walk $W$ with all edges (resp. vertices) being distinct. An $x-x$ walk (resp. trail) is called \textbf{cyclic}.
\end{definition}

\begin{definition}
Let $T \subseteq V(G)$ be a distinguished set of vertices called \textbf{terminals}. A \textbf{$T$-walk} (resp. \textbf{$T$-trail}, \textbf{$T$-path}) is an $x-y$ walk (resp. $x-y$ trail, $x-y$ path) for two distinct $x, y \in T$. (Note that unless explicitly stated otherwise, intermediate vertices of such walks are allowed to be terminals.)
\end{definition}

\begin{definition}
Graph $G$ is called \textbf{inner Eulerian with respect to $T$} (or simply \textbf{inner Eulerian} if $T$ is clear from context) if for any $v \in V(G) - T$ the degree of $v$ in $G$ is even.
\end{definition}

\begin{definition}
Given edge capacities $cap: E(G) \rightarrow \RR_+$, a weighted multiset $\mathcal{P} = \{\alpha_1 \cdot W_1, \dots, \alpha_m \cdot W_m\}$, where $\alpha_i \in \RR_+$ are \textbf{weights} and each $W_i$ is a walk, is said to be a (fractional) \textbf{walk packing} if for any $e \in E(G)$ the \textbf{load} $\mathcal{P}(e) := \sum_i \alpha_i n_i(e)$ of edge $e$ does not exceed $cap(e)$, where $n_i(e) = 0, 1, \ldots$ is the number of occurrences of $e$ in $W_i$.

$\val{\mathcal{P}} := \sum_i \alpha_i$ is called the \textbf{value} of $\mathcal{P}$. If all $\alpha_i \in \ZZ_+$ then $\mathcal{P}$ is called \textbf{integer}.

If $\mathcal{P}, \mathcal{Q}$ are packings and $\alpha \in \RR_+$, $\mathcal{P} + \mathcal{Q}$ denotes a union of weighted multisets and $\alpha \cdot \mathcal{P}$ denotes the result of multiplying all weights in $\mathcal{P}$ by $\alpha$.
\end{definition}

When walks comprising a packing are restricted in some way, the analogous terminology is applied to the packing as a whole. In particular, one may speak of $T$-walk (resp. $T$-trail, $T$-path) packings $\mathcal{P}$ indicating that walks in $\mathcal{P}$ are, in fact, $T$-walks (resp. $T$-trails, $T$-paths).

\begin{definition}
A triple $(G, T, cap)$ consisting of an undirected graph $G$, terminal set $T \subseteq V(G)$ and capacity function $cap: E(G) \rightarrow \RR_+$, is called a \textbf{network}.
\end{definition}

Two notable special cases of constant capacity function to appear throughout our paper are $\1(e) := 1$ and $\2(e) := 2$ for any $e \in E(G)$.

\begin{definition}
Consider network $N = (G, T, cap)$ together with undirected graph $H$ such that $V(H) = T$ (called the \textbf{commodity graph}). A \textbf{multi-commodity flow} (or simply a \textbf{multiflow}) in network~$N$ with commodity graph~$H$ is a $T$-walk packing $\mathcal{P}$ such that for any $T$-walk $W$ in $\mathcal{P}$ its (distinct) endpoints are connected by an edge in $H$.
\end{definition}

We also employ the following graph-theoretic notation:
\begin{definition}
\begin{itemize}
\item Given graph $G$ and $A \subseteq V(G), v \in V(G)$, $\delta(v)$ denotes the set of edges incident to $v$, $\delta(A)$ denotes the set of edges with exactly one endpoint in $A$ and $\gamma(A)$ denotes the set of edges with both endpoints in $A$;
\item For function $f: X \rightarrow \RR$ and $Y \subseteq X$, $f(Y)$ is defined as $\sum_{x \in Y} f(x)$; e.g. for a set of vertices $A$, $f(\delta(A))$ is the total value of $f$ over all edges with exactly one endpoint in $A$;
\item Given edge capacities $cap \colon E(G) \to \RR_+$ in graph $G$ and $S, T \subseteq V(G), S \cap T = \varnothing$, an $S-T$ \textbf{cut} is a vertex set $C$ such that $S \subseteq C \subseteq V(G) - T$; the \textbf{capacity} of cut $C$ is $cap(\delta(C))$; the minimum capacity of an $S-T$ cut is denoted by $\lambda(S, T)$;
\item When graph $G$ is not clear from the context, it is specified explicitly, e.g. $\delta_G$, $\gamma_G$ and $\lambda_G$.
\end{itemize}
\end{definition}


\section{Odd $T$-walk packing algorithm}
\label{sec:symmetric_graph}

Let $(G, T, cap)$ be a network. In this section we introduce an auxiliary network $(\wt{G}, \wt{T}, \wt{cap})$ constructed from $G$ and employ it to provide a strongly polynomial time algorithm for finding a maximum odd $T$-walk packing. This network also plays a crucial role in further sections.

Construct graph $\wt{G}$ with $V(\wt{G}) := V(G) \sqcup V(G)'$, where $V(G)'$ is a disjoint copy of $V(G)$, i.e. each vertex $v \in V(G)$ has its own copy $v' \in V(G)'$, and $E(\wt{G}) := \{u'v, uv' \mid uv \in E(G)\}$. Also, let $v'' := v$ for $v \in V(G)$. If $x, y \in V(G)$, vertices $x$ and $x'$ are called \textbf{symmetric} to each other, and similarly for edges $xy'$ and $x'y$. For a vertex set (resp. an edge set or a walk) $X$, let $X'$ be the vertex set (resp. edge set or walk) consisting of vertices (or edges) symmetric to ones in $X$. Let $\wt{T} := T \sqcup T'$. Finally, define capacities on edges of $\wt{G}$ as $\wt{cap}(uv') := \wt{cap}(u'v) := \frac{1}{2} cap(uv)$ for any $uv \in E(G)$.

The following theorem encapsulates the first of our results announced in \cref{sec:intro}.

\begin{theorem}[Odd $T$-walk packing]
\label{thm:odd_walk_packing}
Given network $(G, T, cap)$, it is possible to construct a maximum fractional odd $T$-walk packing $\mathcal{P}$ in $(G, T, cap)$ in strongly polynomial time.

If all capacities are non-negative even integers, the resulting $\mathcal{P}$ is half-integer. If additionally $cap(\delta(v))$ is divisible by $4$ for all $v \in V(G) - T$, $\mathcal{P}$ is integer.
\end{theorem}
\begin{proof}
Note that $\wt{G}$ is bipartite, so for distinct $x, y \in T$ any $x-y'$ walk in $\wt{G}$ is odd and corresponds to an odd $x-y$ walk in $G$.

Construct commodity graph $H_T$ as follows: $V(H_T) := \wt{T}$ and $E(H_T) := \{t_i t'_j \mid i \neq j\}$. Note that $H_T$ is isomorphic to $K_{|T|,|T|}$ without a perfect matching (see \cref{fig:commodity_graph}).

Consider an arbitrary fractional odd $T$-walk packing $\mathcal{P} = \{\alpha_1 \cdot W_1, \ldots, \alpha_m \cdot W_m\}$ in $(G, T, cap)$ of value $p$. Denote endpoints of $W_k$ as $t_{k,1}$ and $t_{k,2}$; this walk corresponds to a pair of $t_{k,1}-t'_{k,2}$ walk $\wt{W}_k$ and $t'_{k,1}-t_{k,2}$ walk $\wt{W}'_k$ in $\wt{G}$, which are symmetric to each other. Packing $\wt{\mathcal{P}} := \{\alpha_1 \cdot \wt{W}_1, \dots, \alpha_m \cdot \wt{W}_m\}$ in $(\wt{G}, \wt{T}, \wt{cap})$ is of value $p$. For any $xy \in E(G)$ holds $\wt{\mathcal{P}}(xy') + \wt{\mathcal{P}}(x'y) = \mathcal{P}(xy)$ since each occurrence of $xy$ in some walk $W_k$ in $\mathcal{P}$ corresponds to exactly one occurrence of either $x'y$ or $xy'$ in $\wt{W_k}$. The same properties hold for packing $\wt{\mathcal{P}}' = \{\alpha_1 \cdot \wt{W}'_1, \dots, \alpha_m  \cdot \wt{W}'_m\}$, which is the symmetric counterpart of $\wt{\mathcal{P}}$.

Construct $\mathcal{Q} := \frac{1}{2}(\wt{\mathcal{P}} + \wt{\mathcal{P}}')$; the value of $\mathcal{Q}$ is also $p$. For any $xy \in E(G)$ holds $\wt{\mathcal{P}}'(xy') = \wt{\mathcal{P}}(x'y)$, thus $\mathcal{Q}(xy') = \frac{1}{2}(\wt{\mathcal{P}}(xy') + \wt{\mathcal{P}}'(xy')) = \frac{1}{2}(\wt{\mathcal{P}}(xy') + \wt{\mathcal{P}}(x'y)) = \frac{1}{2} \mathcal{P}(xy) \leq \frac{1}{2} cap(xy) = \wt{cap}(xy')$, i.e. $\mathcal{Q}$ is a (self-symmetric) multiflow in $(\wt{G}, \wt{T}, \wt{cap})$ with commodity graph $H_T$. Thus $p$ does not exceed the value of a maximum fractional multiflow in $(\wt{G}, \wt{T}, \wt{cap})$ with commodity graph~$H_T$.

Conversely, consider a fractional multiflow $\mathcal{Q}$ of value $q$ in network $(\wt{G}, \wt{T}, \wt{cap})$ with commodity graph~$H_T$. Construct an odd $T$-walk packing $\mathcal{P}$ of value $q$ in $(G, T, cap)$ by taking preimages of all weighted walks in $\mathcal{Q}$ with their respective weights. Clearly, for $xy \in E(G)$ holds $\mathcal{P}(xy) = \mathcal{Q}(xy') + \mathcal{Q}(x'y) \leq \wt{cap}(xy') + \wt{cap}(x'y) = cap(xy)$. Thus, $q$ does not exceed the value of a maximum fractional odd $T$-walk packing in $(G, T, cap)$.

Therefore, the maximum value of a fractional odd $T$-walk packing in $(G, T, cap)$ equals the value of a maximum fractional multiflow in $(\wt{G}, \wt{T}, \wt{cap})$ with commodity graph $H_T$. To conclude the proof, we utilize the following result due to Karzanov \cite{Kar-94}:
\begin{theorem}
\label{thm:karzanov_anticliques}
Let $(G, T, cap)$ be a network with commodity graph $H$. Denote by $\mathcal{A}$ the family of all inclusion-wise maximal anticliques (i.e. independent sets) in $H$. Suppose $\mathcal{A}$ can be split into two subfamilies $\mathcal{A}_1, \mathcal{A}_2$ such that all anticliques in each family $\mathcal{A}_i$ are pairwise disjoint.

Then a maximum multiflow in $(G,T,cap)$ can be found in strongly polynomial time. If, additionally, $cap$ are integers and $cap(\delta(v))$ is even for any $v \in V(G) - T$, then the resulting multiflow is integer.
\end{theorem}

\begin{figure}
\centering
\includegraphics{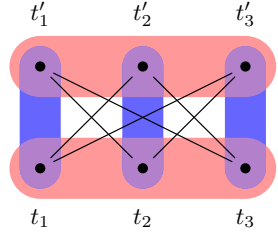}
\caption{Commodity graph $H_T$; anticlique family $\mathcal{A}_1$ is in red and $\mathcal{A}_2$ is in blue.}
\label{fig:commodity_graph}
\end{figure}

Note that the family of anticliques in $H_T$ obeys the property from \cref{thm:karzanov_anticliques}. Indeed, define $\mathcal{A}_1 := \{T, T'\}$ and $\mathcal{A}_2 := \{\{t, t'\} \mid t \in T\}$ (see \cref{fig:commodity_graph}). If some maximal anticlique contains a terminal and its symmetric copy, then it must belong to $\mathcal{A}_2$; otherwise it cannot contain both a vertex from $T$ and a vertex from $T'$, thus it belongs to $\mathcal{A}_1$. Also, if $cap(\delta(v))$ is divisble by~4 for any $v \in V(G) - T$, then $\wt{cap}(\delta(v))$ is even for any $v \in V(\wt{G}) - \wt{T}$. Therefore, applying \cref{thm:karzanov_anticliques} finishes the proof.
\end{proof}


\section{Odd \texorpdfstring{$T$}{T}-walk barrier}
\label{sec:odd_barrier}

In this section we provide a combinatorial description of barrier structure that defines a tight upper bound for the value of a maximum odd $T$-walk packing, which is our second result announced in \cref{sec:intro}. This characterization is surprisingly similar to the corresponding barrier structure for maximum odd $s-t$ path packings due to Schrijver and Seymour~\cite{SS-94}. A strong duality is proven using the equivalence with multiflows from \cref{sec:symmetric_graph}.

\begin{definition}
Given network $(G, T, cap)$, a (not necessarily induced) subgraph $B$ of $G$ with $T \subseteq V(B)$ is called an \textbf{odd $T$-walk barrier} if there is no odd $T$-walk in $B$.

The \textbf{capacity} $cap(B)$ of barrier $B$ is defined as $\frac{1}{2} cap(I(B)) + cap(U(B))$, where for an arbitrary (not necessarily induced) subgraph $H$ of $G$ we use the following notation:
\begin{itemize}
\item $I(H) := \{xy \in E(G) \mid xy \in \delta_G(V(H))\}$ (informally, the edge leaves $H$ and does not return);
\item $U(H) := \{xy \in E(G) \mid x, y \in V(H),\,xy \in E(G) - E(H)\}$ (informally, the edge takes a U-turn by leaving $H$ and immediately returning back).
\end{itemize}

\end{definition}

\begin{figure}
\centering
\includegraphics{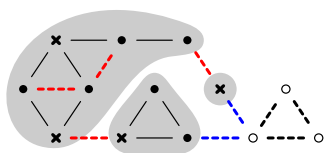}
\caption{An example of an odd $T$-walk barrier $B$; vertices in $T$ are crosses, vertices in $V(B) - T$ are black dots, vertices not in $V(B)$ are white dots; edges in $E(B)$ are solid, edges not in $E(B)$ are dashed; edges in $I(B)$ are blue, edges in $U(B)$ are red.}
\label{fig:barrier}
\end{figure}

It is easy to verify that the capacity of any barrier $B$ is an upper bound for the value of any odd $T$-walk packing $\mathcal{P}$. Indeed, any odd $T$-walk $W$ endowed with weight $\alpha$ in $\mathcal{P}$ is not entirely contained in~$B$; thus it either visits some vertex $v \in V(G) - V(B)$ or traverses some edge $xy \in E(G) - E(B)$ such that $x, y \in V(B)$. In the former case it reserves $\alpha$ units of capacity of at least two edges in $I(B)$, and in the latter case it reserves $\alpha$ units of capacity of at least one edge in $U(B)$. Therefore $\val{\mathcal{P}} \le cap(B)$. The strong duality also holds:

\begin{theorem}[see Appendix]
\label{thm:min_max_odd_walk}
Let $(G, T, cap)$ be a network. If $\mathcal{P}$ ranges over odd $T$-walk packings and $B$ ranges over odd $T$-walk barriers, then $\max\limits_{\mathcal{P}} \val{\mathcal{P}} = \min\limits_{B} cap(B)$.
\end{theorem}

The min-max formula above enables strengthening the statement of \cref{thm:odd_walk_packing} as follows.

\begin{corollary}
Given network $(G, T, cap)$, let $\mathcal{P}$ be a maximum fractional odd $T$-walk packing in $(G, T, cap)$. If all capacities are non-negative even integers, $\val{\mathcal{P}}$ is integer. If additionally $cap(\delta(v))$ is divisible by 4 for all $v \in V(G) - T$, $\val{\mathcal{P}}$ is even integer.
\end{corollary}
\begin{proof}
By \cref{thm:min_max_odd_walk}, $\val{\mathcal{P}} = \frac{1}{2}cap(I(B)) + cap(U(B))$ for minimum odd $T$-walk barrier $B$. If all capacities are even integers, then $cap(U(B))$ and $cap(I(B))$ are also even, therefore the first part of the statement is trivial. Let $A := V(G) - V(B)$, note that $\delta(A) = I(B)$. Under the second condition, note the following congruence:
$$0 \equiv \sum\limits_{\mathclap{v \in A}} cap(\delta(v)) \equiv 2 cap(\gamma(A)) + cap(\delta(A)) \equiv cap(I(B)) \pmod 4$$
Therefore, $\frac{1}{2}cap(I(B))$ is also an even integer.
\end{proof}


\section{Odd $T$-trail packing algorithm}

Hereinafter we focus on network $(G, T, \2)$ for an inner Eulerian graph~$G$. Since all capacities are 2, each edge can be traversed by at most two walks in an integer packing. Our ultimate goal is to construct a maximum integer $T$-trail packing. The third result announced in \cref{sec:intro} is as follows:

\begin{theorem}
\label{thm:odd_trail_packing}
Given network $(G, T, \2)$ with inner Eulerian $G$, it is possible to construct a maximum integer packing of odd $T$-trails in polynomial time. This packing is also a maximum fractional packing of odd $T$-walks in $(G, T, \2)$.
\end{theorem}

We use a chemistry-inspired notation: replace each edge in $G$ with two \textbf{valencies} each of which may be occupied by a walk. More formally:
\begin{definition}
For edge $e \in E(G)$, denote $e^1$ and $e^2$ to be two \textbf{valencies} of $e$; edge~$e$ is called \textbf{underlying} for $e^1$ and $e^2$. Define the \textbf{valence graph} $G^{12}$ to be the graph on the same vertices as $G$ with valencies regarded as edges.
\end{definition}

In what follows, instead of integer odd $T$-walk packings in $(G, T, \2)$ we shall be dealing with integer odd $T$-trail packings in $(G^{12}, T, \1)$, which effectively are sets of edge-disjoint odd $T$-trails in $G^{12}$. However, a $T$-trail in $G^{12}$ may correspond to a non edge-simple $T$-walk in $G$ once we replace valencies with their underlying edges. This is captured as follows:

\begin{definition}
Edge $e \in E(G)$ is called \textbf{irregular} for $T$-trail $W$ in $G^{12}$ if $W$ traverses both valencies $e^1, e^2$ and \textbf{regular} otherwise.
\end{definition}

Hence we are looking for a maximum set of edge-disjoint odd $T$-trails in~$G^{12}$ without irregular edges.

A brief outline of our approach is as follows. In \cref{sec:signing} we introduce \textbf{signing} on valencies that guide $T$-trails and ensure they have proper parities. Given a suitable signing, we prove the existence of an integer odd $T$-trail packing in $(G^{12}, T, \1)$ (with possible irregular edges) of the needed value by reduction to \textbf{bidirected networks}. We also prove that a suitable signing exists.

In \cref{sec:initial_signing} we construct such a signing and also perform the so-called \textbf{terminal evacuation} by introducing an auxiliary terminal $t'$ for each $t \in T$ that is connected to $t$ with a proper number of valencies of certain signs. This transformation allows to assume that no trail in the packing contains any terminal as its intermediate vertex.

Next, \cref{sec:subcubization} ensures that inner vertices are of degree at most~3.

Finally in \cref{sec:regulalization} we deal with irregular edges. We show that whenever both valencies $e^1$ and $e^2$ of some edge $e$ are used by some odd $T$-trail $W$ in the packing, either $W$ could be simplified (preserving its parity) or $e$ is in fact \textbf{redundant}, i.e. $G$ can be reduced by dropping~$e$. This reduction preserves the needed properties of $G$; hence one can recompute a packing and iterate. These iterations continue until there are no more remaining irregular edges.


\subsection{Signed graphs}
\label{sec:signing}

We also use the framework of \textbf{signed graphs} whose edges are endowed with signs ``\texttt{+}'' and ``\texttt{-}''. Intuitively, a signed valence graph introduces a convenient family of $T$-trails defined by the requirement of alternation which makes parity of the trail uniquely determined by the signs of the first and the last valence.

\begin{definition}
A \textbf{signing} is an arbitrary function $M : E(G^{12}) \rightarrow \{\texttt{+}, \texttt{-}\}$. Graph $G^{12}$ together with some signing $M$ forms a \textbf{signed valence graph} $(G^{12}, M)$. In presence of terminal set $T \subseteq V(G)$, a \textbf{signed valence network} $(G^{12}, M, T, \1)$ appears. A $T$-trail $W$ in $(G^{12}, M)$ is called \textbf{alternating} if signs of valencies alternate along $W$.
\end{definition}

\begin{definition}
\label{def:inner_eulerianness}
Signing $M$ is called \textbf{inner balanced} if for any $v \in V(G) - T$, the number of positive edges incident to $v$ equals the number of negative edges incident to $v$.
\end{definition}

We shall need the notion of \textbf{bidirected graphs}, which generalize digraphs and admit three possible kinds of edges: a usual \textbf{directed edge} (ingoing for one endpoint and outgoing for another), a \textbf{positive edge} (which is ingoing for both its endpoints) and a \textbf{negative edge} (which is outgoing for both its endpoints). The definition of a bidirected walk or a bidirected trail is similar to \cref{def:walk_def} with the only difference that for any internal vertex $v_i$ exactly one of $e_i, e_{i+1}$ is ingoing to~$v_i$ and another is outgoing from~$v_i$. Refer to \cite[Ch. 36]{Sch-03} for details.

The notion of inner Eulerianness is extended to bidirected graphs as follows: a bidirected graph~$G$ is \textbf{inner Eulerian} with respect to terminal set $T$ if for any $v \in V(G) - T$ the number of edges ingoing to $v$ is equal to the number of edges outgoing from $v$. Similarly to the undirected case, triple $(G, T, cap)$ consisting of bidirected graph $G$, terminal set $T$ and capacity function $cap$ is called a \textbf{bidirected network}.

We rely on two theorems of a similar kind, one of which is due to Cherkassky~\cite{cher-77} and Lov\'{a}sz~\cite{lov-76}, and another is due to Babenko and Karzanov \cite[Th. 1.1]{BK-07}:

\begin{theorem}[Min-max formula for $T$-trail packings in inner Eulerian undirected graphs \cite{cher-77}, \cite{lov-76}]
\label{thm:lovasz_cherkassky}
Let $(G, T, \1)$ be an inner Eulerian network. Then the value of a maximum packing of $T$-trails equals $\frac{1}{2} \sum_{t \in T} \lambda(\{t\}, T - \{t\})$. Such a packing can be chosen integer and can be constructed in polynomial time.
\end{theorem}

\begin{theorem}[Min-max formula for $T$-trail packings in inner Eulerian bidirected graphs \cite{BK-07}]
\label{thm:babenko_karzanov}
Let $(G, T, \1)$ be an inner Eulerian bidirected network. Then the value of a maximum packing of bidirected $T$-trails equals $\frac{1}{2} \sum_{t \in T} \lambda(\{t\}, T - \{t\})$. Such a packing can be chosen integer and can be constructed in polynomial time.
\end{theorem}

Note that the value of a maximum packing in the latter theorem does not depend on actual directions of edges. As we mentioned before, signings encode a certain family of odd $T$-walks. The following theorem describes why it is important for us.

\begin{theorem}
\label{thm:decomposition}
Let $(G^{12}, M, T, \1)$ be a signed valence network with an inner balanced signing $M$. Let $\mathcal{P}$ be a maximum packing of odd $T$-trails in $(G^{12}, T, \1)$, and $\mathcal{S}$ be a maximum packing of alternating $T$-trails in $(G^{12}, M, T, \1)$. Then $\val{\mathcal{P}} = \val{\mathcal{S}}$. Additionally, $\mathcal{P}$ and $\mathcal{S}$ can be chosen integer and can be constructed in polynomial time.
\end{theorem}
\begin{proof}
Construct an auxiliary bidirected graph $\overleftrightarrow{G^{12}}$ corresponding to the signed valence graph $(G^{12}, M)$ as follows: edges in $\overleftrightarrow{G^{12}}$ correspond to valences in $E(G^{12})$; an edge is positive if the sign of the valence is ``\texttt{+}'' and negative otherwise; $M$ being inner balanced implies that $\overleftrightarrow{G^{12}}$ is inner Eulerian, therefore \cref{thm:babenko_karzanov} is applicable to $(\overleftrightarrow{G^{12}}, T, \1)$. Also note that bidirected $T$-trails in $\overleftrightarrow{G^{12}}$ correspond to alternating $T$-trails in $(G^{12}, M)$. Refer to Figures \ref{fig:signed_packing} and \ref{fig:bidirectional_packing} for an example.

Note that $G^{12}$ is automatically inner Eulerian due to each vertex in $V(G^{12})$ being adjacent to an even  number of valencies, therefore \cref{thm:lovasz_cherkassky} is applicable to $(G^{12}, T, \1)$.

It follows that maximum packing $\mathcal{P}$ of odd $T$-trails in $(G^{12}, T, \1)$ and maximum packing $\overleftrightarrow{\mathcal{P}}$ of bidirected $T$-trails in $(\overleftrightarrow{G^{12}}, T, \1)$ are of the same value $\frac{1}{2} \sum_{t \in T} \lambda(\{t\}, T - \{t\})$. The latter packing can be chosen integer and can be constructed in polynomial time, and then transformed into a maximum integer packing $\mathcal{S}$ of alternating $T$-trails in signed valence network $(G^{12}, M, T, \1)$.
\end{proof}

\begin{figure}
\begin{subfigure}{0.45\textwidth}
\centering
\includegraphics{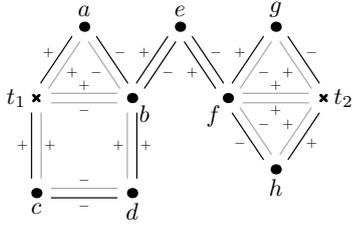}
\caption{Packing of two $t_1-t_2$ trails in $(G^{12}, M, T, \1)$, one of which is even and another is odd.}
\label{fig:signed_packing}
\end{subfigure}\hfill
\begin{subfigure}{0.5\textwidth}
\centering
\includegraphics{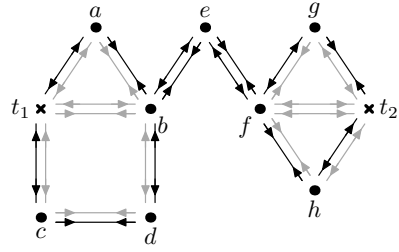}
\caption{Packing of two bidirected $t_1-t_2$ trails in $(\overleftrightarrow{G^{12}}, T, \1)$, one of which is even and another is odd.}
\label{fig:bidirectional_packing}
\end{subfigure}
\caption{Correspondence between signed valence graph with inner balanced signing $(G^{12}, M, T, \1)$ and inner Eulerian bidirected graph $(\overleftrightarrow{G^{12}}, T, \1)$. Terminals are crosses, other vertices are black dots.}
\end{figure}

\begin{definition}
A signed valence network $(G^{12}, M, T, \1)$ with inner balanced signing $M$ is \textbf{$(p, q)$-tight} if 1) there exists an integer $T$-trail packing in $(G^{12}, T, \1)$ of value $p + q$ and 2) the number of ``\texttt{-}'' valencies adjacent to terminals~$T$ is $q$.
\end{definition}

\begin{lemma}
\label{lem:decomposition2}
Given a signed valence network $(G^{12}, M, T, \1)$ with a $(p,q)$-tight inner balanced signing~$M$, there exists an integer packing $\mathcal{P} + \mathcal{Q}$ in $(G^{12}, M, T, \1)$ of value $p + q$, where $\mathcal{P}$ consists of at least~$p$ odd alternating $T$-trails and $\mathcal{Q}$ consists of at most~$q$ even alternating $T$-trails. Moreover, $T$-trails in $\mathcal{P} + \mathcal{Q}$ can be chosen so as to avoid passing through terminals $T$ as intermediate vertices.
\end{lemma}
\begin{proof}
The first tightness property implies existence of an integer $T$-trail packing in $(G^{12}, T, \1)$ of value $p + q$. Then, by \cref{thm:decomposition} we get a packing of $p+q$ alternating $T$-trails in $(G^{12}, M, T, \1)$. Break this packing into two parts $\mathcal{P}$ and $\mathcal{Q}$, where $\mathcal{P}$ consists of odd $T$-trails and $\mathcal{Q}$ consists of even $T$-trails.

The second tightness property implies $\val{\mathcal{Q}} \leq q$ as each trail in $\mathcal{Q}$ has a \texttt{-} valence incident to a terminal, therefore $\val{\mathcal{P}} \ge p$, as needed. 

W.l.o.g. all these $T$-trails do not contain terminals as intermediate vertices (for otherwise, if some alternating $T$-trail $W$ visits $t \in T$ as its intermediate vertex, then $W$ can be split into two subtrails $W_1, W_2$ at $t$; among $W_1$, $W_2$ at least one, say $W_1$ is a valid alternating $T$-trail; replace $W$ with $W_1$ and repeat).
\end{proof}


\subsection{Initial signing and terminal evacuation}
\label{sec:initial_signing}

In this section we present an algorithm for constructing a tight inner balanced signing $M$. This is done with the help of network $(\wt{G}, \wt{T}, \wt{cap})$ from \cref{sec:symmetric_graph}. Note that since $cap = \2$, we have $\wt{cap} = \1$. Degrees of vertices in $\wt{G}$ coincide with degrees of their pre-images in $G$, therefore $\wt{G}$ is also inner Eulerian.

Consider a maximum multiflow $\mathcal{F}$ in $(\wt{G}, \wt{T}, \1)$ with commodity graph $H_T$. \cref{thm:karzanov_anticliques} ensures that $\mathcal{F}$ can be chosen integer, i.e. $\mathcal{F}$ is a collection of edge-disjoint $\wt{T}$-trails (endowed with weight~1) in $\wt{G}$ connecting vertex pairs of form $t_1-t'_2$ for distinct $t_1, t_2 \in T$. Each of these $T$-trails is odd, therefore their pre-images are odd $T$-trails in $G^{12}$ (see \cref{fig:trail_like_components}). Denote the packing of these odd $T$-trails in $(G^{12}, T, \1)$ (taken with weight~1) as~$\mathcal{P}$. Let $p := \val{\mathcal{P}}$. The proof of \cref{thm:odd_walk_packing} implies:

\begin{corollary}
\label{cor:maximum_odd_packing}
$\mathcal{P}$ is a maximum odd $T$-trail packing in $(G^{12}, T, \1)$.
\end{corollary}

Consider the subgraph $\wt{Z}$ of $\wt{G}$ consisting of edges not appearing in $T$-trails of $\mathcal{F}$. Since any vertex $v \in V(\wt{G}) - \wt{T}$ has even degree in $\wt{G}$, $\wt{Z}$ is also inner Eulerian with respect to terminals~$\wt{T}$. Therefore $\wt{Z}$ decomposes into two families of edge-disjoint trails: a collection of cyclic trails and a collection of $\wt{T}$-trails.

The former ones correspond to even cyclic trails in $G^{12}$ (due to biparticity of $\wt{G}$); denote the packing (with unit weights) of these even cyclic trails in $(G^{12}, T, \1)$ as $\mathcal{E}$.

The latter ones may be further subdivided into two categories: (i) $t_1-t_2$ or $t'_1-t'_2$ trails for distinct $t_1, t_2 \in T$; and (ii) $t-t'$ trails for $t \in T$. (Note that $t_1-t_2'$ trails for $t_1 \ne t_2$ cannot appear due to maximality of~$\mathcal{F}$.) The first category corresponds to even $T$-trails in $(G^{12}, T, \1)$. The second category corresponds to odd cyclic trails passing through terminals in $(G^{12}, T, \1)$.  Denote the packings in $(G^{12}, T, \1)$ (with unit weights) corresponding to these two categories as $\mathcal{Q}$ and $\mathcal{R}$ respectively, and let $q := \val{\mathcal{Q}}$ and $r := \val{\mathcal{R}}$.

Note that $\mathcal{P} + \mathcal{Q} + \mathcal{R} + \mathcal{E}$ is an integer packing of (possibly cyclic) trails in $(G^{12}, T, \1)$ that traverses each edge in $G^{12}$ exactly once.

Starting from this moment we forget about graph $\wt{G}$ and release the notation $(\cdot)'$ of its meaning of symmetry in $\wt{G}$.

Perform \textbf{terminal evacuation} as follows. For each terminal $t \in T$ introduce a new terminal $t'$ connected to $t$ by a certain number of edges. Namely, extend each $t_1-t_2$ trail ($t_1$ and $t_2$ may coincide) in $\mathcal{P} + \mathcal{Q} + \mathcal{R}$ with new $t'_1-t_1$ and $t_2-t'_2$ valencies in $G^{12}$, obtaining a new valence graph $G'^{12}$ and new integer packings $\mathcal{P}', \mathcal{Q}', \mathcal{R}'$ in $(G'^{12}, T', \1)$.

Note that originally each terminal $t \in T$ had an even number of adjacent edges in $G^{12}$, therefore it serves as an endpoint for an even number of trails in $\mathcal{P} + \mathcal{Q} + \mathcal{R}$ (counting endpoints of $\mathcal{R}$ twice). Hence, for any $t \in T$ the number of added $t'-t$ valencies is even, therefore the underlying graph $G'$ is well-defined and can be constructed by adding half the number of $t'-t$ valencies. Note that odd (resp. even) $T'$-trails in $G'^{12}$ correspond to odd (resp. even) $T$-trails in $G^{12}$.

Now construct signing $M'$ for $G'^{12}$ by: turning trails in $\mathcal{P'}$ and $\mathcal{R'}$ into odd alternating trails starting and ending with \texttt{+} valencies; turning trails in $\mathcal{Q'}$ into even alternating trails (in any of two possible ways); turning each (cyclic) trail in $\mathcal{E}$ alternating (in any of two possible ways). Clearly, such $M'$ is inner balanced w.r.t. $T'$. Also $M'$ is $(p, q)$-tight. Indeed, $\mathcal{P'} + \mathcal{Q'}$ is a $T'$-trail packing of value $p + q$. Finally, ``\texttt{-}'' valencies adjacent to terminals $T'$ correspond to $T'$-trails in $\mathcal{Q}'$, hence there are exactly $q$ of them. Hence we proved the following theorem.

\begin{theorem}
\label{thm:initial_signing_and_evacuation}
Given network $(G, T, \2)$ with inner Eulerian $G$ such that the maximum value of an odd $T$-walk packing in $(G, T, \2)$ is $p$, it is possible to construct in polynomial time a signed valence network $(G'^{12}, M', T', \1)$ with an inner balanced signing $M'$ such that:
\begin{itemize}
\item $M'$ is $(p, q)$-tight for some $q$;
\item any packing $\mathcal{P}'$ of odd $T'$-trails in $(G', T', \2)$ can be transformed into a packing $\mathcal{P}$ of odd $T$-trails in $(G, T, \2)$ of the same value in polynomial time.
\end{itemize}
\end{theorem}

\begin{figure}
\centering
\includegraphics[scale=0.9]{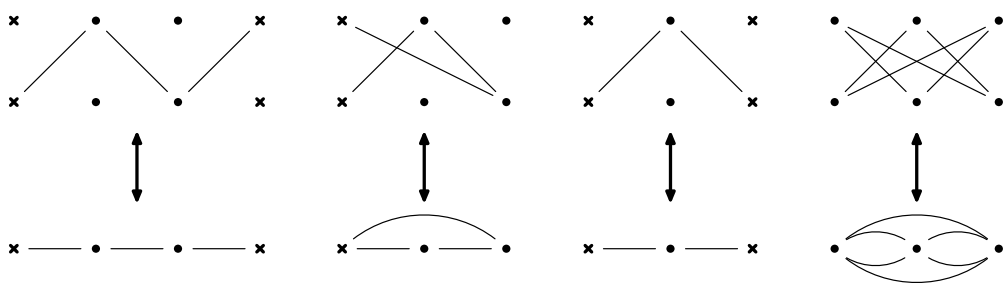}
\caption{Four possible trail-like components of $\wt{G}$ and their images in $G^{12}$: an odd $T$-trail, an odd cyclic trail passing through terminal, an even $T$-trail and an even cyclic trail. Terminals are crosses, other vertices are black dots.}
\label{fig:trail_like_components}
\end{figure}


\subsection{Subcubization}
\label{sec:subcubization}

In this section we prove that it is sufficient to solve the problem only for graphs with degree of non-terminal vertices not exceeding~3, which simplifies the subsequent case splitting.

\begin{definition}
Valence network $(G^{12}, T, \1)$ is called \textbf{inner subcubic}, if $\deg v \leq 3$ for any $v \in V(G) - T$.
\end{definition}

Define the \textbf{supercubicity} of $G$ to be
$$s(G) := \sum\limits_{v \in V(G) - T}{\max\{0, \deg v - 3\}}.$$
Obviously, $s(G) = 0$ for inner subcubic networks.

Let $(G^{12}, M, T, \1)$ be a signed valence network with a $(p, q)$-tight inner balanced signing~$M$. Apply  \cref{lem:decomposition2} to construct an integer packing $\mathcal{P} + \mathcal{Q}$, where $\mathcal{P}$ (resp. $\mathcal{Q}$) consists of at least $p$ (resp. at most $q$) odd (resp. even) alternating $T$-trails.

Consider an inner vertex $v$ of degree $d \geq 4$. Denote edges incident to $v$ in $G$ as $\delta_G(v) = \{e_1, \dots, e_d\}$. Whenever some trail $W$ in $\mathcal{P} + \mathcal{Q}$ passes through $v$, it contains a pair of consequent valencies corresponding to some edges $\{e_i, e_j\}$ in $E(G)$; call $(e_i,e_j)$ for $i \le j$ an (ordered) \textbf{transit pair}. (Note that this ordering of $e_i$ and $e_j$ is not related to the order in which these edges are passed by $W$.) Clearly, whenever an alternating trail passes through a transit pair, it takes valencies of opposite signs.

Valencies corresponding to edges in $\delta_G(v)$ not traversed by any of $W_i$ could also be (arbitrarily) divided into pairs of opposite signs (due to signs balance). Fix some division; it generates (by replacing valencies with their preimages in $G$) more pairs $(e_i,e_j)$ for $i \le j$ that we also regard as transit. Totally we get exactly $d$ transit pairs.

\begin{lemma}
One can partition the set of incident edges $\delta_G(v)$ into two subsets $L \sqcup R$ such that $|L|, |R| \geq 2$ and there are at most two transit pairs $(e_i, e_j)$ (call them \textbf{split transit pairs}) such that $e_i, e_j$ belong to distinct subsets, i.e. $e_i \in L$, $e_j \in R$ or $e_i \in R$, $e_j \in L$.
\end{lemma}
\begin{proof}
Suppose there exists a transit pair $(e_i, e_j)$ for $i < j$. Define $L := \{e_i, e_j\}$ and $R := \delta_G(v) - L$. Each split transit pair must use another valence of $e_i$ or $e_j$, hence there could be at most two such pairs.

On the other hand, if all transit pairs are of the form $(e_i, e_i)$, then an arbitrary partition $L \sqcup R$ with $|L|, |R| \geq 2$ will do.
\end{proof}

\begin{figure}
\centering
\includegraphics{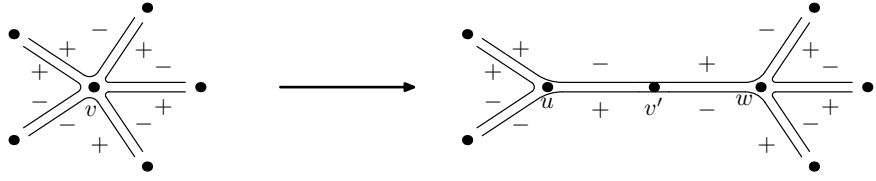}
\caption{Subcubization at $v$}
\label{fig:subcubization}
\end{figure}

Construct a new valence network $(G', M', T' = T, \1)$ (\cref{fig:subcubization}) by replacing vertex $v$ with three vertices $u$, $v'$, $w$ and two edges $uv'$, $v'w$, and also replacing $v$ with $u$ in all edges from $L$ and replacing $v$ with $w$ in all edges from $R$. There are either 0 or 2 split transit pairs; if there are two of them, extend these trails by inserting valencies of $uv'$ and $v'w$ with suitable signs so that signing $M'$ is inner balanced. If there are no transit pairs, simply make both $uv'$ and $v'w$ have one positive and one negative valence. Thus, we also obtain a new packing $\mathcal{P}' + \mathcal{Q}'$ in $(G'^{12}, M', T, \1)$, where $\mathcal{P}'$ (resp. $\mathcal{Q}'$) contains at least~$p$ (resp. at most~$q$) odd (resp. even) alternating $T$-trails not passing through terminals as intermediate vertices.

\begin{lemma}
$s(G') = s(G) - 1$ for the resulting $G'$.
\end{lemma}
\begin{proof}
First of all, $\deg v' = 2 < 3$, so we do not need to consider $v'$ when calculating the change of supercubicity. Then, $\deg u = 1 + |L| \geq 3$ and $\deg w = 1 + |R| \geq 3$; also $\deg u + \deg w = 2 + |L| + |R| = d + 2$ and finally $\deg v = d$. We conclude:
\begin{multline*}
s(G') - s(G) = \max\{0, \deg u - 3\} + \max\{0, \deg w - 3\} - \max\{0, \deg v - 3\} = \\
(\deg u - 3) + (\deg w - 3) - (\deg v - 3) = (\deg u + \deg w) - \deg v - 3 = (d + 2) - d - 3 = -1\,.
\end{multline*}
\end{proof}

Repeat these transformations until there are no more inner vertices with degree more than 3. We obtain an inner subcubic signed valence network $(G'^{12}, M', T' = T, \1)$ with an inner balanced signing $M'$. Note that any $T'$-trail $W'$ in $(G'^{12}, T', \1)$ may easily be transformed into a $T$-trail $W$ in $(G^{12}, T, \1)$ of the same parity by performing all actions in the reverse order and removing added parts of $W'$, if there are any.

\begin{lemma}
The resulting signing $M'$ is $(p, q)$-tight.
\end{lemma}
\begin{proof}
The total number of ``\texttt{-}'' valencies adjacent to terminals does not change during subcubization. Also, packing $\mathcal{P}' + \mathcal{Q}$' has the same value as $\mathcal{P} + \mathcal{Q}$, i.e. $p + q$.
\end{proof}

Hence we proved the following theorem.

\begin{theorem}
\label{thm:subcubization}
If $(G^{12}, M, T, \1)$ is signed valence network with a $(p, q)$-tight inner balanced signing $M$, it is possible to construct a signed valence network $(G'^{12}, M', T' = T, \1)$ with a $(p, q)$-tight inner balanced signing $M'$ such that:
\begin{itemize}
\item $G'$ is inner subcubic;
\item any packing $\mathcal{P}'$ of odd $T'$-trails in $(G', T', \2)$ may be transformed into packing $\mathcal{P}$ of odd $T$-trails in $(G, T, \2)$ of the same value in polynomial time.
\end{itemize}
\end{theorem}


\subsection{Regularization}
\label{sec:regulalization}

Let $(G^{12}, M, T, \1)$ be an inner subcubic signed valence network with a $(p, q)$-tight inner balanced signing $M$. Construct an integer packing $\mathcal{P} + \mathcal{Q}$ of at least $p$ odd alternating $T$-trails (denoted by $\mathcal{P}$) and at most $q$ even alternating $T$-trails (denoted by $\mathcal{Q}$) using \cref{lem:decomposition2}.

Suppose there is edge $xy$ in $E(G)$ that is irregular for some $T$-trail $W$ in $\mathcal{P} + \mathcal{Q}$, i.e. $W$ traverses both of $xy$'s valencies in $G^{12}$. Denote the fragment of $W$ between two occurrences of valencies of $xy$ (but not including them) by~$C$.

Note that $xy$ is not adjacent to any terminal since all $T$-trails in $\mathcal{P} + \mathcal{Q}$ are assumed to avoid passing through terminals as intermediate vertices. Consider cases as follows:

\textbf{Case 1} (\cref{fig:regularization_a}): valencies of $xy$ have opposite signs and $W$ traverses them in the same direction. Simplify $W$ by dropping occurrences of both of these valencies.

\textbf{Case 2} (\cref{fig:regularization_b}): valencies of $xy$ have opposite signs and $W$ traverses them in the opposite directions. Simplify $W$ by dropping occurrences of both of these valencies together with $C$.

\textbf{Case 3} (\cref{fig:regularization_c}): valencies of $xy$ have same signs and $W$ traverses them in the same direction. Simplify $W$ by dropping one of the occurrences of these valencies together with $C$.

\textbf{Case 4} (Figures \ref{fig:regularization_d} and \ref{fig:regularization_e}): valencies of $xy$ have same signs and $W$ traverses them in the opposite directions. Assume that $xy$ is chosen such that $C$ is the shortest possible. W.l.o.g. $y$ belongs to $C$. Finally assume that both valencies of $xy$ are \texttt{+}; the remaining case is done analogously.

\begin{figure}
\centering

\begin{minipage}{\textwidth}
\begin{subfigure}{0.45\textwidth}
\includegraphics{pic-11.mps}
\caption{Case 1}
\label{fig:regularization_a}
\end{subfigure}\hfill
\begin{subfigure}{0.45\textwidth}
\includegraphics{pic-12.mps}
\caption{Case 2}
\label{fig:regularization_b}
\end{subfigure}\\
\begin{subfigure}{0.45\textwidth}
\includegraphics{pic-13.mps}
\caption{Case 3}
\label{fig:regularization_c}
\end{subfigure}\hfill
\begin{subfigure}{0.45\textwidth}
\includegraphics[width=\textwidth]{pic-14.mps}
\caption{Cases 4(i,ii)}
\label{fig:regularization_d}
\end{subfigure}
\end{minipage}\\
\begin{subfigure}{\textwidth}
\includegraphics{pic-15.mps}
\caption{Cases 4(iii,iv)}
\label{fig:regularization_e}
\end{subfigure}
\caption{Regularization cases.}
\label{fig:regularization}
\end{figure}

\begin{lemma}
\label{lem:picture_is_like_we_imagine_it}
In Case 4, $\deg y = 3$ and $C$ starts with a negative valence of some edge $yu$ and terminates with a negative valence of some edge $vy$ with $u \neq v$.
\end{lemma}
\begin{proof}
If $\deg y = 1$, the inner Eulerianess of $y$ is contradicted as both of its adjacent valencies are positive.

If $\deg y = 2$, two valencies of the remaining adjacent edge are both negative and $W$ must follow both of them in order to be alternating. This contradicts the choice of $xy$ with the shortest $C$.

Finally $\deg y = 3$ and $C$ starts with some valence of $yu$ and terminates with some valence of $vy$, both of which are negative. If $u = v$, this would again contradict the choice of $xy$ with the shortest $C$.
\end{proof}

Consider two remaining valencies of $yu$ and $yv$. For signing to be balanced at $y$, one of them must be \texttt{+} and another must be \texttt{-}. Therefore, one of $yu$ and $yv$ has both a \texttt{+} and a \texttt{-} valence; assume it is $yu$, the other case is done analogously. Call $yu$ \textbf{redundant} and obtain a new signed valence network $(G'^{12}, M', T' = T, \1)$ by removing $yu$ in $G'$.

\begin{lemma}
\label{lem:regularization_tightness}
New signing $M'$ is inner balanced and $(p, q)$-tight.
\end{lemma}
\begin{proof}
Signing $M'$ is inner balanced since we remove two valencies of the same edge of opposite signs. Also, the removed edge is not adjacent to a terminal, therefore the total number of ``\texttt{-}'' valencies adjacent to terminals is preserved.

Let us prove that a packing of $T$-trails of value at least $p + q$ still remains. Namely, we alter $\mathcal{P} + \mathcal{Q}$ so that none of its $T$-trails passes through valencies of the removed edge~$yu$.

W.l.o.g. let $e^1$ be the valence of $e = yu$ that is the initial or the final valence of~$C$. Alter $W$ by removing both valencies of $xy$ and $C$, obtaining a (non-alternating) subtrail $W'$ avoiding $e^1$. Note that the remaining valence $e^2$ may either: (i) not belong to any trail in $\mathcal{P} + \mathcal{Q}$; (ii) belong to~$C$; (iii) belong to the same trail~$W$ outside of $C$; (iv) belong to another trail in $\mathcal{P} + \mathcal{Q}$.

In subcases (i,ii) (\cref{fig:regularization_d}) $e^2$ is no longer used by any trail in $\mathcal{P} + \mathcal{Q}$; replace $W$ with $W'$.
In subcases (iii,iv) (\cref{fig:regularization_e}), consider trail containing $e^2$ and replace $e^2$ in it with the $y-u$ fragment of $C$ that is different from $e^1$ (note that $C$ is not used by $W'$ anymore).
\end{proof}

Repeat the procedure until no more irregular edges exist. In Cases 1--3 the signed valence graph does not change, but the total length of odd $T$-trails in the packing decreases. Therefore, this step may be iterated until either there are no irregular edges or Case~4 happens and we obtain a new signed graph $(G'^{12}, M')$. In other words, $(|E(G)|, L)$, where $L$ is the total length of $T$-trails in $\mathcal{P}$, decreases lexicographically in each case. Thus the total number of iterations is polynomial. Let us summarize the result of this section by the following theorem.
\begin{theorem}
\label{thm:regularization}
If $(G^{12}, M, T, \1)$ is an inner subcubic signed valence network with a $(p, q)$-tight inner balanced signing $M$, then it is possible to construct an integer packing of odd $T$-trails in $(G, T, \2)$ of value at least $p$ in polynomial time.
\end{theorem}


\subsection{Concluding the proof}

\begin{proof}[Proof of \cref{thm:odd_trail_packing}]

Let $(G, T, \2)$ be a inner Eulerian network and let $p$ be the value of a maximum odd $T$-walk packing in it. Apply \cref{thm:initial_signing_and_evacuation} to construct a signed valence network $(G'^{12}, M', T', \1)$ with a $(p,q)$-tight inner balanced signing $M'$ (for some $q$). By \cref{thm:subcubization} the latter network can be replaced by a subcubic signed valence network $(G''^{12}, M'', T'', \1)$ with a $(p,q)$-tight inner balanced signing $M''$. Now \cref{thm:regularization} implies the existence of packing $\mathcal{P}''$ of odd $T''$-trails of value $p$ in $(G'', T'', \2)$.

Finally reverse the changes applied to network: $\mathcal{P}''$ gives rise to packing $\mathcal{P}'$ of odd $T'$-trails of the same value~$p$ in $(G', T', \2)$ (by \cref{thm:subcubization}); in its turn, $\mathcal{P}'$ generates packing $\mathcal{P}$ of odd $T$-trails of value~$p$ in $(G, T, \2)$ (by \cref{thm:initial_signing_and_evacuation}), as needed.

Note that all of the above steps take polynomial time.
\end{proof}


\newpage

\bibliography{main}

\newpage


\appendix

\section{Proof of \cref{thm:min_max_odd_walk}}

We are going to use the following min-max relation for multiflows due to Karzanov under the same commodity graph constraints as in \cref{thm:karzanov_anticliques}; see~\cite{Kar-94}:

\begin{theorem}
\label{thm:min_max_multicommodity}
Let $(G, T, cap)$ be a network and $H$ be a commodity graph obeying the conditions of \cref{thm:karzanov_anticliques}. Consider a partition $\mathcal{X} := \{X_1, X_2, \dots, X_k\}$ of $T$ into disjoint nonempty sets $X_1, X_2, \ldots, X_k$; call $\mathcal{X}$ \textbf{proper} if each $X_i$ is an independent set of $H$. Define the \textbf{capacity} of $\mathcal{X}$ as $cap(\mathcal{X}) := \frac{1}{2} \sum_{i=1}^k \lambda(X_i, T - X_i)$; also let $Y_i$ ($X_i \subseteq Y_i \subseteq V(G) - (T - X_i)$) be the corresponding minimum cut between $X_i$ and $T - X_i$ in $G$.

Let $\mathcal{F}$ range over multiflows in network $(G, T, cap)$ with commodity graph $H$, and $\mathcal{X}$ range over proper partitions of $T$; then
$$\max_{\mathcal{F}} \val{\mathcal{F}} = \min_{\mathcal{X}} cap(\mathcal{X}).$$
Moreover, the minimum $\mathcal{X}$ can be chosen such that $Y_i$ are pairwise disjoint.
\end{theorem}

Let us rewrite the capacity of an odd $T$-walk barrier as follows.
\begin{definition}
Let $H$ be a (not necessarily induced) subgraph of $G$. Define function $S[H]: E(G) \rightarrow \{0, \frac{1}{2}, 1\}$ called a \textbf{slice of $H$} by $S[H](e) := 1$ for $e \in U(H)$, $S[H](e) := \frac{1}{2}$ for $e \in I(H)$, and $0$ otherwise.
\end{definition}
A similar notion of slices (differing by a factor of~2) earlier appeared in \cite{SS-94}. Now for an odd $T$-walk barrier $cap(B) = cap \bigcdot S[B]$, where $\bigcdot$ stands for the scalar product.

\begin{lemma}
Let $H$ be a (not necessarily induced) subgraph of $G$ and $\mathcal{C}$ be the family of connected components of $H$. Then $S[H] = \sum_{C \in \mathcal{C}} S[C]$ (regarded as functions on $E(G)$).
\end{lemma}
\begin{proof}
If $S[H](xy) = 1$, either $x$ and $y$ belong to the same connected component $C_1$, in which case $S[C_1](xy) = 1$ and $S[C_2](xy) = 0$ for any other $C_2 \in \mathcal{C}$, $C_2 \neq C_1$, or to two distinct connected components $C_1, C_2 \in \mathcal{C}$, in which case $S[C_1](xy) + S[C_2](xy) = \frac{1}{2} + \frac{1}{2} = 1$ and $S[C_3](xy) = 0$ for any other $C_3 \in \mathcal{C}$, $C_3 \neq C_1, C_2$.

If $S[H](xy) = \frac{1}{2}$, let $C_1$ be the connected component containing one of $xy$'s endpoints; in this case $S[C_1](xy) = \frac{1}{2}$ and $S[C_2](xy) = 0$ for $C_2 \in \mathcal{C}$, $C_2 \neq C_1$.

Otherwise $S[H](xy) = 0$ and $S[C_1](xy) = 0$ for any $C_1 \in \mathcal{C}$.
\end{proof}

\begin{lemma}
\label{lem:min_max_leq}
Let $\mathcal{X}$ range over proper partitions of $\wt{T}$ and $B$ be fixed odd $T$-walk barrier, then
$$\min_{\mathcal{X}} cap(\mathcal{X}) \leq cap(B).$$
\end{lemma}
\begin{proof}

Consider some barrier $B$ and let $\mathcal{C}$ denote the family of connected components of~$B$. We construct a proper partition $\mathcal{X}$ of $\wt{T}$ and a corresponding family of cuts separating $X$ and $\wt{T} - X$ for any $X \in \mathcal{X}$. Moreover, the capacities of these cuts are bounded by corresponding summands in $\sum_{C \in \mathcal{C}} cap \bigcdot S[C]$.

Call a component $C \in \mathcal{C}$ \textbf{redundant} if it contains no terminals from $T$ and \textbf{singular} if it contains a single terminal. Otherwise $C$ must be \textbf{bipartite} (regarded as a graph) and all terminals must belong to the same part of bipartition so that there is no odd $T$-walk within this component. Construct a proper partition as follows.

If $C$ is singular containing just terminal $t \in T$, enclose $t$ with its symmetric vertex $t'$ with set $X := \{t, t'\}$, which is a maximal anticlique in $H_T$. Note that $Y := C \cup C'$ is a cut between $X$ and $\wt{T} - X$; edges of $\delta(Y)$ in $\wt{G}$ correspond to edges of $I(C)$ in $G$. Thus
\begin{multline}
\label{eqn:ineq_singular}
\frac{1}{2} \wt{cap}(\delta(Y)) = \frac{1}{2} \sum\limits_{\mathclap{xy \in \delta(Y)}} \wt{cap}(xy) = \frac{1}{2}\sum\limits_{\mathclap{xy \in I(C)}}\left(\wt{cap}(xy') + \wt{cap}(x'y)\right) = \\
= \frac{1}{2} \sum\limits_{\mathclap{xy \in I(C)}}cap(xy) = \frac{1}{2} cap(I(C)) \leq cap \bigcdot S[C].
\end{multline}

If $C$ is bipartite containing multiple terminals $t_1, \dots, t_k$, introduce sets $X := \{t_1, \dots, t_k\}$ and $X' := \{t'_1, \dots, t'_k\}$, which are subsets of maximal anticliques $T$ and $T'$ in $H_T$, respectively. Let $L$ and $R$ be the bi-partition parts of $C$ such that $t_1, \dots, t_k \in L$. Note that $Y := L \sqcup R'$ is a cut between $X$ and $\wt{T} - X$ and, symmetrically, $Y' = L' \sqcup R$ is a cut between $X'$ and $\wt{T} - X'$.

There are two kinds of edges in $\delta(Y) \cup \delta(Y')$ in $\wt{G}$: the ones connecting $Y$ or $Y'$ with $V(G) - (Y \sqcup Y')$, and the ones connecting $Y$ with $Y'$. Again, the former ones correspond to edges in $I(C)$, while the latter ones are edges connecting $L$ with $L'$ or $R$ with $R'$; therefore their pre-images belong to $\gamma(L)$ or $\gamma(R)$, hence they belong to $U(C)$. Using these observations, we get the following:
\begin{multline}
\label{eqn:ineq_bipartite}
\frac{1}{2}\left(\wt{cap}(\delta(Y)) + \wt{cap}(\delta(Y'))\right) = \frac{1}{2}\sum\limits_{\mathclap{xy \in \delta(Y \sqcup Y')}} \wt{cap}(xy) + \sum\limits_{\mathclap{x \in Y,\,y \in Y'}} \wt{cap}(xy) \leq \\
\leq \frac{1}{2} \sum\limits_{\mathclap{xy \in I(C)}} (\wt{cap}(xy') + \wt{cap}(x'y)) + \sum\limits_{\mathclap{xy \in U(C)}} (\wt{cap}(xy') + \wt{cap}(x'y)) = \\
= \frac{1}{2} \sum\limits_{\mathclap{xy \in I(C)}} cap(xy) + \sum\limits_{\mathclap{xy \in U(C)}} cap(xy) = \frac{1}{2} cap(I(C)) + cap(U(C)) = cap \bigcdot S[C].
\end{multline}

Finally, if $C$ is redundant, it does not produce any set for our proper partition. Note that each vertex $v \in \wt{T}$ belongs to exactly one of the formed sets in our partition.

Summing \cref{eqn:ineq_singular} and \cref{eqn:ineq_bipartite} over all singular and bipartite components, we get $cap(\mathcal{X}) \leq cap \bigcdot S[B] = cap(B)$, which completes the proof.
\end{proof}

\begin{figure}
\centering
\begin{subfigure}{\textwidth}
\centering
\includegraphics{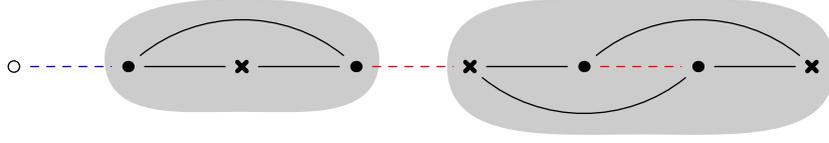}
\caption{Odd $T$-walk barrier $B$ in $(G, T, c)$: left connected component is singular; right connected component is non-singular; blue and red edges are $I(B)$ and $U(B)$, respectively.}
\end{subfigure}
\begin{subfigure}{\textwidth}
\centering
\includegraphics{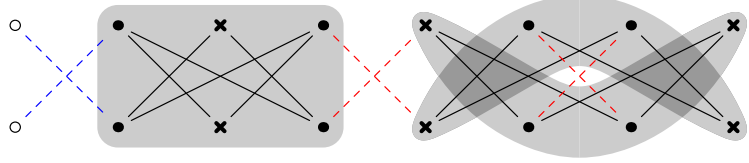}
\caption{Proper partition $\mathcal{X}$ in $(\wt{G}, \wt{T}, \wt{c})$: left set is singular; two right sets are non-singular and symmetric; blue and red edges are accounted in $cap(\mathcal{X})$ with weight $\frac{1}{2}$ and 1, respectively.}
\end{subfigure}
\caption{Odd $T$-walk barrier $B$ and its corresponding proper partition $\mathcal{X}$.}
\end{figure}

\begin{lemma}
There exists a proper partition $\mathcal{X}$ with minimum $cap(\mathcal{X})$ such that no two $X_1, X_2 \in \mathcal{X}$ are subsets of the same maximal anticlique in $H_T$.
\end{lemma}
\begin{proof}
Construct another proper partition $\mathcal{Z}$ by replacing $X_1$ and $X_2$ with $X_1 \sqcup X_2$. The change in capacity is $cap(\mathcal{Z}) - cap(\mathcal{X}) = \lambda(X_1 \sqcup X_2, \wt{T} - (X_1 \sqcup X_2)) - \lambda(X_1, \wt{T} - X_1) - \lambda(X_2, \wt{T} - X_2)$.

Let $Y_i$ be a minimum cut between $X_i$ and $\wt{T} - X_i$ for $i = 1,2$. Submodularity of cut capacities \cite[Sec.~44.1a]{Sch-03} implies $\lambda(X_1, \wt{T} - X_1) + \lambda(X_2, \wt{T} - X_2) = cap(\delta(Y_1)) + cap(\delta(Y_2)) \geq cap(\delta(Y_1 \cap Y_2)) + cap(\delta(Y_1 \cup Y_2)) \geq 0 + \lambda(X_1 \sqcup X_2, \wt{T} - (X_1 \sqcup X_2))$. Therefore, $cap(\mathcal{Z}) \leq cap(\mathcal{X})$. Repeat merging parts until done.
\end{proof}

From the previous lemma trivially follows the following corollary.

\begin{corollary}
There exists a proper partition $\mathcal{X}$ with minimum $cap(\mathcal{X})$ such that all of its sets are of the form $\{t, t'\}$ for $t \in T$ (call them \textbf{singular}) except for, possibly, two symmetric sets $X, X' \in \mathcal{X}$ (call them \textbf{non-singular}).
\end{corollary}

\begin{lemma}
There exists a proper partition $\mathcal{X}$ with minimum $cap(\mathcal{X})$ such that the associated minimum cuts between $X$ and $\wt{T} - X$ for $X \in \mathcal{X}$ obey the following properties:
\begin{itemize}
\item if $X$ is singular, the corresponding cut $Y$ is self-symmetric, i.e. $Y' = Y$;
\item if $X$ and $X'$ are non-singular, their corresponding cuts are also symmetric to each other;
\item all the above cuts are disjoint.
\end{itemize}
\end{lemma}
\begin{proof}
Using \cref{thm:min_max_multicommodity}, we may choose $\mathcal{X}$ such that the corresponding cuts $Y$, $Y'$ are disjoint.

If $Y$ is a minimum cut between $\{t, t'\}$ and $\wt{T} - \{t, t'\}$ for some $t \in T$, then so is $Y'$. Since by submodularity $cap(\delta(Y)) + cap(\delta(Y')) \geq cap(\delta(Y \cap Y')) + cap(\delta(Y \cup Y'))$, it follows that $Y \cap Y'$ is also a minimum cut between $\{t, t'\}$ and $\wt{T} - \{t, t'\}$; also $Y \cup Y'$ is self-symmetric by construction. Replace $Y$ with $Y \cap Y'$.

The second part of the statement is similar. If $X$ and $X'$ are non-singular and $Y_1$ (resp.~$Y_2$) is a minimum cut between $X$ and $\wt{T} - X$ (resp. $X'$ and $\wt{T} - X'$), then $Y'_2$ and (resp. $Y'_1$) is also a minimum cut for the same vertex sets. Using a similar argument, replace $Y_1$ and $Y_2$ with $Y_1 \cap Y'_2$ and $Y_2 \cap Y'_1$, which are also minimum cuts symmetric to each other.

The steps above replace cuts with their subsets; therefore the last property is preserved.
\end{proof}

\begin{lemma}
\label{lem:min_max_geq}
Let $\mathcal{X}$ range over proper partitions of $T$ and $B$ range over odd $T$-walk barriers, then
$$\min\limits_{\mathcal{X}} cap(\mathcal{X}) \geq \min\limits_{B} cap(B).$$
\end{lemma}
\begin{proof}
Pick $\mathcal{X}$ with the minimum capacity $cap(\mathcal{X})$ satisfying the properties from the previous lemma.

If $X = \{t, t'\} \in \mathcal{X}$ is singular and $Y$ is the corresponding cut, consider the pre-image $C \subseteq V(G)$ of $Y$. Add the subgraph of $G$ induced by $C$ to barrier $B$. Note that $U(C) = \varnothing$. Now

\begin{multline}
\label{eqn:ineq_singular_1}
cap \bigcdot S[C] = \frac{1}{2} cap(I(C)) = \frac{1}{2} \sum\limits_{xy \in I(C)}cap(xy) = \\
= \frac{1}{2} \sum\limits_{xy \in I(C)}(\wt{cap}(xy') + \wt{cap}(x'y)) = \frac{1}{2} \wt{cap}(\delta(Y)).
\end{multline}

If $X, X' \in \mathcal{X}$ are non-singular, consider their corresponding cuts $Y$ and $Y'$. Let $Y = L \sqcup R'$ for $L, R \subseteq V(G)$ (which implies $Y' = L' \sqcup R$). Add the bipartite subgraph $D$ of $G$ induced by $L$ and $R$ to barrier $B$. Note that $U(D)$ is the pre-image of edges $xy'$ with $x \in Y$ and $y' \in Y'$. Therefore

\begin{multline}
\label{eqn:ineq_bipartite_2}
cap \bigcdot S[D] = cap(U(D)) + \frac{1}{2}cap(I(D)) = \sum\limits_{xy \in U(D)}cap(xy) + \frac{1}{2}\sum\limits_{xy \in I(D)}cap(xy) = \\
= 2 \cdot \frac{1}{2} \sum\limits_{xy \in U(D)}(\wt{cap}(xy') + \wt{cap}(x'y)) + \frac{1}{2} \sum\limits_{xy \in I(D)}(\wt{cap}(xy') + \wt{cap}(x'y)) = \\
= \frac{1}{2}(\wt{cap}(\delta(Y)) + \wt{cap}(\delta(Y'))).
\end{multline}

\end{proof}

\begin{proof}[Proof of \cref{thm:min_max_odd_walk}]
From \cref{thm:odd_walk_packing} we know that $\max_{\mathcal{P}} \val{\mathcal{P}} = \max_{\mathcal{Q}} \val{\mathcal{Q}}$, where $\mathcal{P}$ ranges over odd $T$-walk packings and $\mathcal{Q}$ ranges over multiflows in $(\wt{G}, \wt{T}, \wt{cap})$ with commodity graph $H_T$. Then, from \cref{thm:min_max_multicommodity} it follows that $\max_{\mathcal{Q}} \val{\mathcal{Q}} = \min_{\mathcal{X}} cap(\mathcal{X})$, where $\mathcal{X}$ ranges over proper partitions of $T$ in $H_T$. Finally, from \cref{lem:min_max_leq} and \cref{lem:min_max_geq} it follows that $\min_{\mathcal{X}} cap(\mathcal{X}) = \min_{B} cap(B)$.
\end{proof}


\end{document}